\documentclass[11pt,letterpaper,headinclude,footinclude,fleqn,reqno]{amsart}                 
\usepackage[T1]{fontenc}                   
\usepackage[utf8]{inputenc}                 
\usepackage[english]{babel}       
\usepackage{graphicx}                      % 
\usepackage[font=small]{quoting}            %  
\usepackage{caption}  
\usepackage{amsmath}

\usepackage[top=1in,bottom=1.5in,left=1.2in,right=1.2in]{geometry}
\usepackage{verbatim}
\usepackage{color}
\usepackage{picture}
\usepackage{graphicx}
\usepackage{graphics}
\usepackage{comment}
\usepackage{hyperref}
\hypersetup{colorlinks,linkcolor={blue},citecolor={blue},urlcolor={red}}

\newcommand{\bt}{\bar{t}}

\def\n{\nabla}
\def\<{\langle}
\def\>{\rangle}

\def\S{\Sigma}
\def\n{\nabla}

\def\n{\nabla}

\def\p{\partial}

\def\k{\kappa}

\def\s{\sigma}

\def\n{\nabla}
\def\<{\langle}
\def\>{\rangle}
\def\div{{\rm div}}

\def\n{\nabla}

\def\RR{\mathbb{R}}

\def\SS{\mathbb{S}}

\def\p{\partial}

\def\s{\sigma}

\def\R{\mathbb{R}}

\def\V{\mathcal V}
\def\L{\mathcal{L}}
\def\bt{\mathbf{b}_\theta}

\newtheorem{thm}{Theorem}[section] 
 
\newtheorem{cor}[thm]{Corollary} 
\newtheorem{prop}[thm]{Proposition} 
 
\newtheorem{defn}[thm]{Definition} 
\theoremstyle{definition} 
 
\theoremstyle{remark} 

\usepackage{graphicx}

\numberwithin{equation}{section}

\makeatletter
\@namedef{subjclassname@2020}{\textup{2020} Mathematics Subject Classification}
\makeatother

\begin{document}
\setlength{\baselineskip}{1.2\baselineskip}

\title[A constrained mean curvature flow and Alexandrov-Fenchel inequalities]
{A constrained mean curvature flow and Alexandrov-Fenchel inequalities}
	
\author{Xinqun Mei}

\address{Mathematisches Institut, Albert-Ludwigs-Universit\"{a}t Freiburg, Freiburg im Breisgau, 79104, Germany}

\address{School of Mathematical Sciences, University of Science and Technology of China, Hefei, 230026, P.R.China}

\email{xinqun.mei@math.uni-freiburg.de}

\author{Guofang Wang}
\address{Mathematisches Institut, Albert-Ludwigs-Universit\"{a}t Freiburg, Freiburg im Breisgau, 79104, Germany}
\email{guofang.wang@math.uni-freiburg.de}

\author{Liangjun Weng}
\address{School of Mathematical Sciences, Anhui University, Hefei, 230601, P. R. China}
\email{ljweng08@mail.ustc.edu.cn}
	
	% Please provide minimum  5 keywords.
 
	\maketitle
	\begin{abstract}
		In this article, we study a locally constrained mean curvature flow for star-shaped hypersurfaces with capillary boundary in the half-space. We prove its long-time existence and the global convergence to a spherical cap. Furthermore, the capillary quermassintegrals defined in \cite{WWX2022} evolve monotonically along the flow, and hence we establish a class of new Alexandrov-Fenchel inequalities for convex hypersurfaces with capillary boundary in the half-space.
	\end{abstract}
	
	\bigskip
	
%\tableofcontents	
	\section{Introduction}
	
The mean curvature flow plays an important role in geometric analysis and differential geometry and has been extensively studied. One of  the 
classical results proved by  Huisken \cite{Hui1984} states that it contracts a closed convex  hypersurface into a round point and converges to a sphere after a suitable rescaling. Later he constructed a normalized  mean curvature flow to prove the isoperimetric inequality in  \cite{Hui1987}.  Inspired by the Minkowski formulas for closed hypersurfaces,
Guan-Li \cite{GL15} introduced a new locally constrained mean curvature flow to study the isoperimetric problem, which is defined by
\begin{eqnarray}\label{guan-li flow}
	 	\partial_tx=(n -\<x,\nu\>H)\nu,  \qquad  \quad \hbox{ for }  x:M\times [0,T)\to\RR^{n+1}, 
\end{eqnarray}
where $M$ is a $n$-dimensional closed manifold, $H$ and $\nu$ are the mean curvature and unit outward normal vector of hypersurface $x(M,\cdot)$ respectively. Along  flow \eqref{guan-li flow},
 the enclosed domain has a fixed volume, while the
area is monotone decreasing by the Minkowski formulas. Guan-Li obtained the long-time existence of this flow and proved that it smoothly converges to a round
sphere if the initial hypersurface is star-shaped. 
As a result, it also yields a flow proof of classical Alexandrov-Fenchel inequalities of quermassintegrals in convex geometry.
For more details
see \cite[Corollary 1.2]{GL15}.
 Using similar ideas and  new locally constrained curvature flows, there have been a lot of work developed  in the last decades. See for instance \cite{CGLS,ChS,GL09,GL18,GLW,HL,HLW,SX19} and references therein.

 In this paper, we are interested in the isoperimetric inequalities and  Alexandrov-Fenchel inequalities for  capillary hypersurfaces in the half-space. We first  introduce a new locally constrained mean curvature type flow which is suitable  for  capillary hypersurfaces in the half-space. The flow  is  also motivated by the Minkowski formula (i.e. \eqref{minkowski formula-0})  for capillary hypersurfaces in $\bar \RR^{n+1}_+$.	Besides, this paper is inspired by a new research on capillary hypersurfaces in \cite{WWX2022}, where the authors establish a  class of new Alexandrov-Fenchel inequalities for capillary hypersurfaces in the half-space.
 
  By abusing a little bit the terminology, a capillary hypersurface is defined in this paper as the following.
\begin{defn}
	A hypersurface in $\bar{{\R}}_+^{n+1}$  with boundary supported on $\p\bar{{\R}}_+^{n+1}$ is called capillary hypersurface if
	it intersects with $\p\bar{{\R}}_+^{n+1}$ at a constant angle $\theta\in (0,\pi)$.
\end{defn}
Let $\Sigma_{t}$ be a family of hypersurfaces with boundary in $\bar{\RR}^{n+1}_{+}$ given by a family of isometric embeddings $x(\cdot, t): M \rightarrow \bar {\RR}^{n+1}_{+}$ from a compact $n$-dimensional manifold $M$ with   boundary $\partial M$ ($n\geq 2$) such that 
\begin{eqnarray}
	{\rm 	int}(\Sigma_{t})=x\left({\rm int}(M), t\right)\subset\mathbb{\RR}^{n+1}_{+},\quad \partial\Sigma_{t}=x(\partial M,t)\subset \partial \mathbb{\RR}_{+}^{n+1}.\notag
\end{eqnarray}
In this paper, we study the following locally constrained mean curvature flow 
for capillary hypersurfaces  defined by %$x(\cdot, t)$ satisfy

\begin{eqnarray}\label{flow with capillary}
	\left\{ \begin{array}{rcll}
		(\partial_t x)^\perp  &=&\left(n +n\cos\theta\<\nu, e\>-H\<x, \nu\> \right) \nu , \quad &
		\hbox{ in }M\times[0,T),\\
		\<\nu ,e \>& =& -\cos\theta 
		\quad & \hbox{ on }\partial M \times [0,T),\\
		 x(\cdot,0)  &=& x_0(\cdot) \quad & \text{ in }   M,
	\end{array}\right. 
\end{eqnarray}	  
 where  $\nu, H$ are the unit normal, the mean curvature of $\S$ respectively, $e:=-E_{n+1}=(0,\cdots, 0,-1)$ is the unit outward normal of $\p\bar\RR^{n+1}_+$ in $\bar\RR^{n+1}_+$.

 Compared to flow \eqref{guan-li flow}, there is an extra term $n\cos\theta \<\nu,e\>$ in the speed function $f$ of our flow \eqref{flow with capillary}, which comes naturally from the  Minkowski formula, namely 
		\begin{eqnarray}\label{minkowski formula-0}
n\int_{\Sigma}\left(1+\cos\theta\<\nu, e\>\right)dA=\int_{\Sigma}H\<x, \nu\>dA,
\end{eqnarray}
for any capillary hypersurfaces.  For a proof we refer to  \cite{AS, WWX2022}.
There is an interesting  family of capillary hypersurfaces, which are the spherical caps lying entirely in $\bar{\RR}^{n+1}_{+}$ and intersecting $\partial{\RR}^{n+1}_{+}$ with a constant contact angle $\theta\in (0,\pi)$  given by
\begin{eqnarray}\label{sph-cap}
	C_{\theta, r}:=\Big\{x\in \bar{\RR}^{n+1}_{+}\big||x-r\cos\theta e|=r \Big\},~r~\in [0, \infty).
\end{eqnarray}
 
One can easily check that $C_{\theta,r}$ is  a static solution to flow \eqref{flow with capillary}, that is 
\begin{eqnarray}\label{eq_a1}
	1+\cos\theta\<\nu, e\>-\frac{H}{n}\<x, \nu\>=0.
\end{eqnarray} 
In fact,  $C_{\theta, r}$ consist of all solutions of \eqref{eq_a1}. See Proposition \ref{prop_0} below.
We also denote $\SS_\theta^n$, $\widehat {\p \SS_\theta^n}$, $ {\mathbb B}^{n+1}_\theta$ as
$$ \SS^n_\theta: =\{ x\in \SS^n \,|\,  \langle x, E_{n+1} \rangle >\cos \theta\}, \quad {\mathbb B}^{n+1}_\theta :=\{ x\in   {\mathbb B}^{n+1} \,|\,  \langle x, E_{n+1} \rangle >\cos \theta\},
$$
$$ { \widehat{\p \SS^n_\theta}:= \{ x\in   {\mathbb B}^{n+1} \,|\,  \langle x, E_{n+1} \rangle =\cos \theta\}.
}
$$
 
Now we state one of the main theorems in this paper.	
\begin{thm}\label{thm 1.1}
	If the initial capillary hypersurface is   star-shaped with respect to origin in $\bar \RR^{n+1}_+$ and the contact angle  $\theta\in(0,\pi)$,
	 then  flow \eqref{flow with capillary} exists for all time. Moreover, 
	$x(\cdot, t)$ smoothly converges to a uniquely determined spherical cap $C_{\theta, r_0}$ around $e$, as $t\to\infty$.		
\end{thm}
%In particular, when $\theta=\frac \pi 2$, this result was proved previous by Guan-Li  \cite[Therorem 1.1]{GL15}. 
It is worth noting that  we could obtain the uniform a priori estimates for flow \eqref{flow with capillary} in the whole range $\theta\in(0,\pi)$ instead of only a  partial range of contact angle compared with results in  \cite{WW20} or \cite{WeX21}.

It is easy to see that flow \eqref{flow with capillary} preserves the enclosed volume 
$$\V_{0,\theta}(\widehat{\Sigma}):= |\widehat{ \S}|,$$
by using \eqref{minkowski formula-0}.
A nice feature of flow \eqref{flow with capillary} is that along the flow, the capillary  area (one can refer to \cite{Finn,RR} for more interpretation)
$$\V_{1,\theta}(\widehat{\Sigma}):= \frac{1}{n+1} ( |\S|-\cos\theta |\widehat{\p\S}|),$$
is decreasing, which follows from
 the higher order Minkowski formulas, proved recently in 
 \cite[Proposition 2.5]{WWX2022}, %generalize \eqref{minkowski formula-0} to the high order counterpart, which reads
\begin{eqnarray}\label{minkowski formula}
	\int_{\Sigma}H_{k-1}\left(1+\cos\theta\<\nu, e\>\right)dA= \int_{\Sigma}H_{k}\<x, \nu\>dA,
\end{eqnarray}
where $H_k$  ($1\le k\le n$) is the normalized $k$-th mean curvature of $\Sigma\subset \R^{n+1}$ and $H_0=1$.

As a result, Theorem \ref{thm 1.1} yields a flow proof for capillary isoperimetric inequality for star-shaped hypersurfaces in $\bar \RR^{n+1}_+$. 
	
%As a result, we give a flow proof for capillary isoperimetric inequality for star-shaped capillary hypersurface in half space when $\theta\in (0,\pi)$.
\begin{cor}
	For $n\ge 2$,  let $\Sigma\subset    \bar{\mathbb{R}}^{n+1}_+$ be a star-shaped capillary
	hypersurface with a contact angle  $\theta \in (0, \pi)$, 
%	hypersurfaces with boundary on $\p \R^{n+1}_+$, 
then 
	there holds
	\begin{eqnarray}\label{capillary iso ineq}
\frac {|\Sigma|-\cos \theta |\widehat{\p \Sigma}|}
{|{\mathbb S}_{\theta}^{n}|-\cos \theta |  { \widehat{\p \SS^n_\theta}}|}\ge  \left( \frac {|\widehat \Sigma|}{|{\mathbb B}^{n+1}_\theta|} \right)^{\frac n {n+1}}.
	\end{eqnarray} Moreover, equality holds if and only if $\Sigma$  is a spherical cap in \eqref{sph-cap}. \end{cor}

The capillary isoperimetric inequality holds true in fact for any hypersurface in $\bar\R^{n+1}_+$. For the proof we refer to 
   \cite[Chapter 19]{Maggi}, using the spherical symmetrization.

%Furthermore, 
Recently, as the generalization  of the capillary  area $\V_{1,\theta}(\widehat{\S})$, a class of new  quermassintegrals $\V_{k,\theta}(\widehat{\S})$ for capillary hypersurfaces $\S\subset \bar \RR^{n+1}_+$  have been introduced in  \cite{WWX2022}, which  are, for $1\leq k\leq n$,
\begin{eqnarray}\label{quermassintegrals}
	\V_{k+1,\theta}(\widehat{\S})  
	&:=&
	\frac{1}{n+1}\left(\int_\S H_kdA - \frac{\cos\theta \sin^k\theta  }{n}\int_{\p\S} H_{k-1}^{\p\S}ds \right)
	,
\end{eqnarray}where  $H_{k-1}^{\p\S}$ is the normalized  $(k-1)$-th mean curvature of $\p\S\subset \RR^n$.   Another motivation to 
study  flow \eqref{flow with capillary} is to establish  the following Alexandrov-Fenchel  inequalities.

\begin{thm}\label{thm1.4}
	For $n\ge 2$,  let $\Sigma\subset    \bar{\mathbb{R}}^{n+1}_+$ be a  convex capillary hypersurface with a contact angle $\theta \in (0, \frac{{\pi}}{2} ]$, then 
there holds
\begin{eqnarray}\label{af ineq V_k V_0}
 %\left(\frac{g(n,\theta)}{n+1}\right)^{\frac{n-k}{n-k+1}}
\left( \frac{  \V_{k,\theta}(\widehat{\S}) }{\textbf{b}_\theta} \right)^{\frac{1}{n+1-k}}\geq \left(\frac{\V_{0,\theta}(\widehat{\S})}{\textbf{b}_\theta} \right)^{\frac{1}{n+1}},  \quad  \forall \,  1\le k <n,
\end{eqnarray}where $\bt$ is the $(n+1)$ dimensional  volume of ${\mathbb B}^{n+1}_\theta$, % $|{\mathbb B}^{n+1}_\theta|$,
with equality if and only if $\Sigma$  is a spherical cap in \eqref{sph-cap}. 	
\end{thm}
Note that \eqref{af ineq V_k V_0}  gives  partially an affirmative answer to 
  a question in \cite[Conjecture 1.5]{WWX2022}, where the authors proved
  \begin{eqnarray*}%\label{af ineq V_k V_0}
  %\left(\frac{g(n,\theta)}{n+1}\right)^{\frac{n-k}{n-k+1}}
  \frac{  \V_{n,\theta}(\widehat{\S}) }{\textbf{b}_\theta} \geq \left( \frac{  \V_{k,\theta}(\widehat{\S}) }{\textbf{b}_\theta} \right)^{\frac{1}{n+1-k}}  
   %\left(\frac{\V_{0,\theta}(\widehat{\S})}{\textbf{b}_\theta} \right)^{\frac{1}{n+1}}
   ,  \quad  \forall \,  1\le k <n,
\end{eqnarray*}
  by using a locally constrained inverse curvature type flow.
%In particular, when $\theta=\frac{\pi}{2}$, inequality \eqref{af ineq V_k V_0} can be derived using the reflection argument along $\p\RR^{n+1}_+$ and Alexandrov-Fenchel inequality  for closed convex hypersurface in $\RR^{n+1}$ (cf. \cite[Corollary 1.2]{GL15}).  
When $ k=1$ in \eqref{af ineq V_k V_0}, it is  the capillary isoperimetric inequality considered above. %for hypersurface in $\bar\RR^{n+1}_+$ (cf. \cite[Chapter 19]{Maggi}). 
We remark that the restriction of range $\theta\in(0,\frac{\pi}{2}]$ is only used in showing the convexity preserving property along the flow \eqref{flow with capillary}.  See   Proposition \ref{H case preserve convexity} or Eq. \eqref{deri of Widetilde H}. 

As a direct consequence of Theorem \ref{thm1.4},  for $k=2$ in \eqref{af ineq V_k V_0}, we obtain a volumetric Minkowski inequality for convex
capillary hypersurfaces in $\bar\RR^{n+1}_+$.
\begin{cor}\label{cor1.4}
	For $n\ge 2$,  let $\Sigma\subset    \bar{\mathbb{R}}^{n+1}_+$ be a  {convex} capillary hypersurface with a contact angle $ {\theta \in (0, \frac{{\pi}}{2} ]}$, then 
	there holds
	\begin{eqnarray}\label{mink ineq}
 \int_\S HdA  -\sin\theta\cos\theta |\p\S| \geq n(n+1) \bt^{\frac{2}{n+1}}|\widehat{\S}|^{\frac{n-1}{n+1}},
	\end{eqnarray} 
	with equality if and only if $\Sigma$  is a spherical cap in \eqref{sph-cap}. 	
\end{cor}
It is remarkable to note that Agostiniani-Fogagnolo-Mazzieri  obtained a volumetric Minkowski inequality in \cite[Theorem 1.5]{AFM2} without any convexity assumption for closed hypersurfaces in Euclidean space, see also \cite{CW,Glaudo,Qiu} for the relevant results. We expect  that  \eqref{mink ineq}  holds true  for general capillary hypersurfaces without any convexity assumption by replacing  $\int_\S H dA$ with $\int_\S |H|dA$. 

We end  the introduction by mentioning some interesting results on curvature flows with free or capillary boundaries in the Euclidean space. The classical mean curvature flow with free boundary was studied by Stahl in \cite{Sta1}, which showed that evolving a strictly convex, embedded hypersurface in the interior of a ball or half-space by mean curvature flow with free boundary produces a Type I singularity which, after rescaling, is asymptotically hemispherical. See also \cite{Buck,EHIZ} for  further analysis of the behavior of singularities that develop on the free boundary of the evolving hypersurface. Lambert-Scheuer \cite{LS16} proved that the inverse mean curvature flow for hypersurfaces with free boundary in a ball would drive a strictly convex hypersurfaces into a flat disk. Recently,  Wang-Xia \cite{WX20} studied a locally constrained mean curvature type flow for embedded hypersurface with free boundary in the unit ball and showed that  it converges to a spherical cap, which motivated by a new Minkowski formula in \cite{WX19}. See also \cite{WW20} for the capillary case.  Furthermore, Scheuer-Wang-Xia \cite{SWX} introduced a new inverse  nonlinear curvature flow for embedded hypersurfaces with free boundary in the ball and obtained a class of new Alexandrov-Fenchel inequalities after obtaining  that flow converges to a spherical cap. See also \cite{WWX2022,WeX21} for the very recent development of hypersurfaces with capillary boundary.

\vspace{.2cm}
 \textit{The rest of the article is structured as follows.} In Section \ref{sec2},  some notations  and known results   about capillary  hypersurfaces in the half space are introduced, including the relevant evolution equations. 
In Section \ref{sec3}, we first obtain the uniform estimates for flow \eqref{flow with capillary} and  prove its long-time existence and smooth
 convergence in Section \ref{sec3.1}, then we  complete the proof of Theorem \ref{thm 1.1}. In Section \ref{sec3.2}, we prove the convexity preserving along flow \eqref{flow with capillary}, which implies  the  Alexandrov-Fenchel inequalities, that is, Theorem \ref{thm1.4}.
\vspace{.2cm}

	\section{Capillary hypersurface in half-space}\label{sec2}
   Let  $\S\subset\bar{\RR}^{n+1}_+$ be a smooth, properly embedded  capillary hypersurface, given by the embedding $x:M\to \bar{\RR}^{n+1}_+$, where $M$ is a compact, orientable smooth manifold of dimension $n$ with non-empty boundary. 
   If there is no confusion, we do not distinguish $\Sigma$ 
 and the embedding $x$.
   Let $\mu$ be the unit outward co-normal of $\p\S$ in $\S$ and  $\overline{\nu}$ be the unit normal to $\partial\Sigma$ in $\partial\mathbb{R}^{n+1}_+$ such that $\{\nu,\mu\}$ and $\{\overline{\nu},e\}$ have the same orientation in normal bundle of $\partial\Sigma\subset \bar{\mathbb{R}}^{n+1}_+$, where $e:=-E_{n+1}$. We define 
the contact angle $\theta$ between  the hypersurface $\Sigma$ and the support $\partial\bar{\R}^{n+1}_+$ by
$$\< \nu, e \>=\cos (\pi-\theta).$$
% $ \nu $ and $\overline N\circ x$  are defined by $(\pi-\theta)$ in Definition \ref{capillary bdry},
It  follows  
\begin{eqnarray}\label{co-normal bundle}
	\begin{array} {rcl}
		e &=&\sin\theta \mu-\cos\theta \nu,
		\\
		\overline{\nu} &=&\cos\theta \mu+\sin\theta \nu.
	\end{array}
\end{eqnarray}
We use $D$ to denote the Levi-Civita connection of $\bar{\RR}^{n+1}_+$ w.r.t the Euclidean metric $\delta$, 
and $\n$  the Levi-Civita connection on $\S$ w.r.t the induced metric $g$ from the immersion $x$. 
The operator $\div, \Delta$, and $\n^2$ are  the divergence, Laplacian, and Hessian operator on $\S$ respectively.
The second fundamental form $h$ of $x$  is defined by
$$D_{X}Y=\n_{X}Y- h(X,Y)\nu.$$ 
Denote $\k=(\k_1, \k_2,\cdots, \k_n)$ be the set of principal curvatures, i.e, the set of eigenvalues of $h$. For other notations, we following the paper \cite{WWX2022}.

The second fundamental form of $\p \S$ in  $\RR^n$ is given by $$\widehat{h}(X, Y):= -\<\nabla^{\RR^n}_X Y, \bar \nu\>= -\<D_X Y, \bar \nu\>, \quad X, Y\in T(\p\S).$$
The second equality holds since $\<\bar \nu, \bar N\circ x\>=0$.
The second fundamental form of $\p \S$ in $\S$ is given by
$$\widetilde{h}(X, Y):= -\<\nabla_X Y, \mu\>= -\<D_X Y, \mu\>, \quad X, Y\in T(\p\S).$$
The second equality holds since $\<\nu, \mu\>=0$. Besides, $\widehat{h}, \widetilde{h}$  and $h$ have some nice relationships, cf. \cite[Proposition 2.3]{WWX2022}.

	Let $\Sigma_{t}$ be a family of smoothly embedded hypersurfaces with $\theta$-capillary boundary in $\bar{\RR}^{n+1}_{+}$, given by the embeddings $x(\cdot, t): M\rightarrow \bar{\RR}^{n+1}_{+}$, which evolves by the general flow
	\begin{eqnarray}\label{general flow}
		\partial_{t}x=f\nu+T,
	\end{eqnarray}
	with  some normal speed function $f$  and $T\in  T\Sigma_{t}$.  In this paper, $f$ consists mainly of mean curvature, and hence \eqref{general flow} is a mean curvature type flow, where $T$ is added such that flow preserves the capillary condition. 
	
Along flow \eqref{general flow}, we have the following evolution equations for the induced metric $g_{ij}$, % the area element $dA_t$, 
the unit outward normal $\nu$, the second fundamental form $h_{ij}$, the Weingarten
matrix  $h^i_j$,  the mean curvature $H$ of the hypersurfaces $\Sigma_t$.  These evolution equations have been shown in \cite[Proposition 2.11]{WeX21} and will be used in deriving the a priori estimates.	
	\begin{prop}[\cite{WeX21}]\label{basic evolution eq}
		Along flow \eqref{general flow}, it holds that
		\begin{enumerate} 
			\item $\p_t g_{ij}=2fh_{ij}+\n_i T_j+\n_jT_i$.
		%	\item $\p_tdA_t =\left(fH+\div(T)\right)dA_t.$ 
			\item $\p_t\nu =-\n f+h(e_i,T)e_i$.	
			\item $\p_t h_{ij}=-\n^2_{ij}f +fh_{ik}h_{j}^k +\n_T h_{ij}+h_{j}^k\n_iT_k+h_{i}^k\n_j T_k.$
			\item $\p_t h^i_j=-\n^i\n_{j}f -fh_{j}^kh^{i}_k+\n_T h^i_j.$
			\item $\p_t H=-\Delta f-|h|^2 f+ \n_T H$.
		\end{enumerate}
	\end{prop}

\section{A priori estimate}\label{sec3}

In this section, we obtain the uniform a priori estimates for flow \eqref{flow with capillary}.

\subsection{Further evolution equations}  We need evolution equations  for various geometric quantities.
 For simplicity, we introduce the linearized operator with respect to \eqref{flow with capillary} as
\begin{eqnarray*}
	\mathcal{L}:=\partial_{t}-\<x, \nu\>\Delta -\<T+Hx-n\cos\theta e, \nabla\>.
\end{eqnarray*}
We derive the evolution equations for various geometric quantities
under flow \eqref{flow with capillary}. That is,  the normal speed function in \eqref{general flow} is chosen to be $$f:=n\left(1+\cos\theta\<\nu, e\>\right)-H\<x, \nu\>.$$
	\begin{prop}
		Along flow \eqref{flow with capillary},   the support function $u=\<x, \nu\>$ satisfies
		\begin{eqnarray}\label{evo of u}
			\mathcal{L}u= n+n\cos\theta\<\nu, e\>-2uH +u^{2}|h|^{2},
		\end{eqnarray}
		and 
		\begin{eqnarray}\label{deri of u}
			\nabla_{\mu}u=\cot\theta h(\mu, \mu)u, \quad {\rm on}~\partial\Sigma_{t}.
		\end{eqnarray}
	\end{prop}
	\begin{proof}
		Direct computation yields
		\begin{eqnarray*}
			\nabla_{i}u=h_{ik}\<x, e_{k}\>,
		\end{eqnarray*}
		and
		\begin{eqnarray*}
			\nabla_{ij}u= h_{ij;k}\<x, e_{k}\>+h_{ij}-u(h^{2})_{ij},
		\end{eqnarray*}
	then  
		\begin{eqnarray*}
			\Delta u=\<x, \nabla H\>+H-u|h|^{2}.
		\end{eqnarray*}
		In view of  these formulas  and Proposition \ref{basic evolution eq}, we have
		\begin{eqnarray*}
			\partial_{t}u&=&\<\p_{t}x, \nu\>+\<x, \partial_{t}\nu\>=f-\<x, \nabla f\>+h(x^{T}, T)\\
			&=&\left(n+n\cos\theta\<\nu, e\>-uH\right)-n\cos\theta \left<x, \nabla \<\nu, e\>\right>+u\<x, \nabla H\>\\
			&&+\<Hx, \nabla u\>+h(x^{T}, T)\\
			&=&\<x, \nu\>\Delta u-n\<\cos\theta e, \nabla u \>+\<Hx, \nabla u\>+ \left(n+n\cos\theta\<\nu, e\>-2uH\right)\\
			&&+u^{2}|h|^{2}+\<T, \nabla u\>,
		\end{eqnarray*}
	that is,
		\begin{eqnarray*}
			\mathcal{L}u= n+n\cos\theta\<\nu, e\>-2uH +u^{2}|h|^{2}.
		\end{eqnarray*}
		On $\partial\Sigma_{t}$, since $\mu$ is a principal direction of $\S$ by \cite[Proposition 2.3]{WWX2022}, together with  \eqref{co-normal bundle}, it yields
		\begin{eqnarray*}
			\nabla_{\mu}\<x, \nu\>=\<x, h(\mu, \mu)\mu\>=\cos\theta h(\mu, \mu)\<x, \bar{\nu}\>=
			\cot\theta h(\mu, \mu)\<x, \nu\>.
		\end{eqnarray*}
	\end{proof}
For later use, now we introduce a new function.
	\begin{defn} We call  the function defined by
	\begin{eqnarray}\label{relative support}
		\bar u=\frac{\<x, \nu\>}{1+\cos\theta \<\nu,e\>},
	\end{eqnarray} 
the	{\rm capillary support function} of capillary hypersurface $\Sigma$ in $\bar\RR^{n+1}_+$.
	\end{defn}
	
It is interesting to see that the capillary support function has nice properties.
\begin{prop}\label{evo of relative u}
	Along   flow \eqref{flow with capillary},  the capillary support function  $\bar u$ satisfies
\begin{eqnarray}\label{evo of bar u}
\L \bar u=n-2\bar u H+\bar u^2 |h|^2+ 2u \left\<\n \bar u,\frac{\n (1+\cos\theta \<\nu,e\>)}{1+\cos\theta \<\nu,e\>} \right\>,
\end{eqnarray}
and  \begin{eqnarray}\label{neumann of bar u}
	\n_{\mu} \bar u=0, \quad \text{ on } \p\S_t.
\end{eqnarray}  
\end{prop}	

\begin{proof}
From the Codazzi formula, we have	\begin{eqnarray}\label{eq_b}
		\langle \nu,e\rangle_{;kl}= h_{kl;s}\langle e_s,e\rangle -(h^2)_{kl}\langle \nu,e\rangle ,
	\end{eqnarray}
	which yields
	\begin{eqnarray}
		\Delta\<\nu,e\>= \<\n H,e\>-|h|^2 \<\nu,e\>.
	\end{eqnarray}
Combining  it with
\begin{eqnarray*}
	\<\n f,e\>=n\cos\theta h(e^T,e^T)-\<x,\nu\>\< \n H,e\>-Hh(x^T,e^T),
\end{eqnarray*}
we obtain
\begin{eqnarray*}
 \p_t \<\nu,e\>&=& -\<\n f,e\>+h(e_i,T)\<e_i,e\>
 \\&=&-n\cos\theta h(e^T,e^T)+u\<\n H,e\>+Hh(x^T,e^T)+h(T,e^T)
 \\&=& u\left( \Delta \<\nu,e\>+|h|^2\<\nu,e\>\right)-n\cos\theta h(e^T,e^T) +Hh(x^T,e^T)+h(T,e^T).
\end{eqnarray*}
Thus we conclude that
\begin{eqnarray}\label{ev of nu e}
	\L \<\nu,e\>= u|h|^2 \<\nu,e\>.
\end{eqnarray}%\begin{eqnarray}\boxed{ 	\L (1+\cos\theta\<\nu,e\>)=u  |h|^2 (1+\cos\theta\<\nu,e\>)- u|h|^2}.\end{eqnarray}
Recall $\bar u:=\frac{u}{1+\cos\theta \<\nu,e\>}$.   \eqref{evo of u} and \eqref{ev of nu e} yield
\begin{eqnarray*}
	\L \bar u&=& \frac{1}{1+\cos\theta \<\nu,e\>} \L u -\frac{\<x,\nu\>}{(1+\cos\theta\<\nu,e\>)^2} \L (1+\cos\theta \<\nu,e\>) \\&&-\<x,\nu\>\left(\frac{2u|\n(1+\cos\theta \<\nu,e\>)|^2}{(1+\cos\theta\<\nu,e\>)^3 }-2\frac{ \<x,\nu_{;i}\> \cos\theta \<\nu_{;i},e\>}{(1+\cos\theta\<\nu,e\>)^2}\right)
	\\&=&n-\frac{2uH}{1+\cos\theta \<\nu,e\>}+\frac{u^2 |h|^2}{(1+\cos\theta\<\nu,e\>)^2}+ 2u \left\<\n \bar u,\frac{\n (1+\cos\theta \<\nu,e\>)}{1+\cos\theta \<\nu,e\>} \right\>	.%\\&\geq & \left( \frac{u|H|}{\sqrt{n}(1+\cos\theta \<\nu,e\>)} -\sqrt{n}\right)^2,
\end{eqnarray*}
Since $\mu$ is a principal direction of $\S$ by \cite[Proposition 2.3]{WWX2022},
we have on $\p \Sigma_t$ %\begin{eqnarray*} 	\n_\mu \langle x,\nu\rangle  =\langle x, h(\mu,\mu)\mu\rangle =\cos \theta h(\mu,\mu)\langle x, \overline \nu\rangle =  \cot\theta h(\mu,\mu)\langle x,  \nu\rangle 	,\end{eqnarray*} and
		\begin{eqnarray}\label{deri of nu e}
	\nabla_{\mu}\left(1+\cos\theta\<\nu, e\>\right)=-\sin\theta h(\mu, \mu)\<\nu, e\>=\cot\theta h(\mu,\mu) (1+\cos\theta \<\nu,e\>),
\end{eqnarray} which, together with \eqref{deri of u}, implies \begin{eqnarray*}
	\n_{\mu} \bar u=0.
\end{eqnarray*}
	\end{proof}	
Now we compute the evolution equation for the mean curvature $H$.
	\begin{prop}\label{evol of H}
		Along flow \eqref{flow with capillary}, we have
		\begin{eqnarray}\label{evo of H}
			\mathcal{L}H=2\<\n H, \n u\>-n|h|^{2}+H^{2},
		\end{eqnarray}
		and
		\begin{eqnarray}\label{deri of H}
			\nabla_{\mu}H=0,\quad{\rm on }~\partial\Sigma_{t}.
		\end{eqnarray}
	\end{prop}
	\begin{proof}
	From \eqref{eq_b}
	%	\begin{eqnarray*}
	%		\<\nu, e\>_{;ij}=h_{ij;k}\<e_{k}, e\>-(h^{2})_{ij}\<\nu, e\>,
	%	\end{eqnarray*}
	it is easy to see
		\begin{eqnarray*}
			f_{;ij}&=&n\cos\theta h_{ij;k}\<e_{k}, e\>-n\cos\theta (h^{2})_{ij}\<\nu, e\>-H_{;ij}\<x, \nu\>-H_{;i}h_{jk}\<x, e_{k}\>\\
			&&-H_{;j}h_{ik}\<x, e_{k}\>-Hh_{ij}-Hh_{ij;k}\<x, e_{k}\>+H(h^{2})_{ij}\<x, \nu\>.
		\end{eqnarray*}
	Combining it with Proposition \ref{basic evolution eq} (6), we have
		\begin{eqnarray*}
			\partial_{t}H&=&-\Delta f-f|h|^{2}+\<\nabla H, T\>\\
			&=&-\<n\cos\theta e, \nabla H\>+n\cos\theta\<\nu, e\>|h|^{2}+\<x, \nu\>\Delta H
			+2H_{;i}h_{ik}\<x, e_{k}\>\\
			&&+\<Hx, \nabla H\>-uH|h|^{2}-(n+n\cos\theta\<\nu, e\>-Hu)|h|^{2}\\
			&&+\<\nabla H, T\>+H^{2}\\
			&=&\<x, \nu\>\Delta H+\<Hx+T-n\cos\theta e, \nabla H\>+2H_{;i}h_{ik}\<x, e_{k}\>-n|h|^{2}+H^{2},
		\end{eqnarray*}
	that is,
		\begin{eqnarray*}
			\mathcal{L}H=2\<\n H,\n u\>-n|h|^{2}+H^{2}.
		\end{eqnarray*}
		It has been shown in \cite[Proposition 4.3]{WWX2022}, on $\partial\Sigma_{t}$, 
		\begin{eqnarray*}
			\n_\mu f=\cot\theta h(\mu,\mu) f.
		\end{eqnarray*}
Combining with \eqref{deri of u} and \eqref{deri of nu e}, %}, we have\begin{eqnarray*}	\n_\mu (f+\langle x,\nu\rangle )=\cot\theta h(\mu,\mu) (f+\langle x,\nu\rangle ).\end{eqnarray*}
altogether implies
\begin{eqnarray*}
	\n_\mu H&=&\n_\mu \left(\frac{n(1+\cos\theta \langle \nu,e\rangle )-f}{\langle x,\nu\rangle }\right) =0. 
\end{eqnarray*}
	\end{proof} 
If  a hypersurface is strictly convex, we let  $(b^{ij})$ denote the inverse of $(h_{ij})$, and  set $$\bar H:=\sum_{i=1}^n \frac{1}{\k_i}=g_{ij}b^{ji},$$
which is the harmonic curvature. The harmonic curvature satisfies the following evolution equation,
together a nice boundary inequality, provided that $\theta \in (0, \frac \pi 2]$. It will be used to prove that the flow preserves the convexity.
 	\begin{prop}\label{evo of bar H}
		Along flow \eqref{flow with capillary}, $\bar H$ satisfies
		\begin{eqnarray}\label{evo of widetilde H}
	\mathcal{L} \bar H=-2u (b^{ii}) ^2b^{jj}h_{ij;k}^2-2b^{i}_kH_{;k}\<x,e_i\>-u|h|^2 \bar H  -H\bar H+n H\<x,\nu\>+n^2.
		\end{eqnarray}
	For $\theta\in (0,\frac{\pi}{2}]$, it holds
		\begin{eqnarray}\label{deri of Widetilde H}
			\nabla_{\mu}\bar{H}\leq 0, \quad {\rm on}~\partial\Sigma_{t}.
		\end{eqnarray}
	\end{prop}
	\begin{proof}
  Direct computation yields
		\begin{eqnarray*}
			\nabla_{k}\bar{H}=-b^{is}\n_k h_{st} b^t_i,%=-(\widetilde{h}^{ii})^{2}h_{ii;k},
		\end{eqnarray*}
		and
		\begin{eqnarray}\label{Laplacian of bar H}
			\Delta\bar{H}&=&-b^{is}\Delta h_{st}b^t_i+2b^{ip} \n^k h_{pq}b^{qs} \n_k h_{st}b^t_i.
		\end{eqnarray}
The Ricci equation and the Codazzi   equation  yield\begin{eqnarray*}
		h_{kl;ii}&=&h_{ki;li}=h_{ki;il}+R^p_{ili}h_{pk}+R^p_{kli}h_{pi}
		\\&=&h_{ii;kl}+(h_{pl}H-h_{pi}h_{li})h_{pk}+(h_{pl}h_{ki}-h_{pi}h_{kl})h_{pi}
		\\&=&H_{;kl}+h_{pk}h_{pl}H-|h|^2 h_{kl},
	\end{eqnarray*} and hence
	\begin{eqnarray*}
	 H_{ ;ij}=\Delta (h_{ij })-H(h^2)_{ij}+|h|^2h_{ij}.
	\end{eqnarray*}
Together with \eqref{Laplacian of bar H}, it implies	\begin{eqnarray*}
b^i_kb^l_i	H_{ ;kl}&=&b^i_kb^l_i\big(\Delta (h_{kl })-H(h^2)_{kl}+|h|^2h_{kl}\big)
\\&=&  -\Delta \bar H+2(b^{ii})^2 b^{jj}h_{ij;k}^2  -nH+|h|^2 \bar H.
\end{eqnarray*}
On the other hand,
\begin{eqnarray*}
	 \p_t \bar H&=&\p_t b^i_i=-b^{i }_k\p_t h_{l}^k b^l_i
	 \\&=&-b^{i}_k\left(-\n^k\n_{l}f -fh_{l}^ph^{k}_p+\n_T h^k_l\right) b^l_i
	 \\&=&b^i_k b^l_i \n^k\n_l f+nf-b^i_k \n_T h^k_l b^l_i.
\end{eqnarray*}
Using again
\begin{eqnarray*}
 \langle x,\nu\rangle_{;kl}&=  h_{kl}+h_{kl;s}\langle x,e_s\rangle -(h^2)_{kl}\langle x,\nu\rangle ,
\end{eqnarray*}
and
\begin{eqnarray*}
 \langle \nu,e\rangle_{;kl}= h_{kl;s}\langle e_s,e\rangle -(h^2)_{kl}\langle \nu,e\rangle ,
\end{eqnarray*}
we obtain
\begin{eqnarray*}
	\n^k\n_l f&=& \n^k\n_l (n+n\cos\theta \<\nu,e\>-H\<x,\nu\>)
	\\&=&n\cos\theta  (h_{kl;s}\langle e_s,e\rangle -(h^2)_{kl}\langle \nu,e\rangle)-H_{;k}\<x,h_{ls}e_s\>-H_{;l}\<x,h_{ks}e_s\> \\&&-H_{;kl}\<x,\nu\>-H(h_{kl}+h_{kl;s}\langle x,e_s\rangle -(h^2)_{kl}\langle x,\nu\rangle),
\end{eqnarray*}
and
\begin{eqnarray*}
	\p_t \bar H&=&b^i_kb^l_i\Big[ n\cos\theta  \Big(h_{kl;s}\langle e_s,e\rangle -(h^2)_{kl}\langle \nu,e\rangle\Big)-H_{;k}\<x,h_{ls}e_s\>-H_{;l}\<x,h_{ks}e_s\> \\&&-H_{;kl}\<x,\nu\>-H\big(h_{kl}+h_{kl;s}\langle x,e_s\rangle -(h^2)_{kl}\langle x,\nu\rangle\big) \Big] +n^2(1+\cos\theta \<\nu,e\>)\\&&-nH\<x,\nu\> + \<T,\n\bar H\>
	\\&=& -u\big(  -\Delta \bar H+2(b^{ii})^2 b^{jj}h_{ij;k}^2  -nH+|h|^2 \bar H\big)
	 -n\cos\theta \<e,\n \bar H\>\\&&-n^2\cos\theta \<\nu,e\> -2b^i_kH_{;k}\<x,e_i\>    -H \bar H+H\<x,\n \bar H\>+nH\<x,\nu\>
\\&&+n^2(1+\cos\theta \<\nu,e\>)-nH\<x,\nu\> + \<T,\n\bar H\>.
\end{eqnarray*}
It  follows
\begin{eqnarray*}
	\mathcal{L} \bar H %&=&  \p_t\bar H -\<x,\nu\>\Delta \bar H-\<T+Hx-n\cos\theta e,\n\bar H\>
=-2u (b^{ii}) ^2b^{jj}h_{ij;k}^2-2b^{i}_kH_{;k}\<x,e_i\>-u|h|^2 \bar H  -H\bar H+n H\<x,\nu\>+n^2.
\end{eqnarray*}
Along $\p\S_t$, choosing an orthonormal frame $\{e_\alpha\}_{\alpha=2}^{n}$ of $T{\p\S_t}$ such that $\{e_1:=\mu,(e_\alpha)_{\alpha=2}^n\}$ forms an orthonormal frame for $T\S_t$, from \cite[Proposition 2.3 (3)]{WWX2022}, we know
		\begin{eqnarray*}
			h_{\alpha\alpha;\mu}=\cot\theta h_{\alpha\alpha}(h_{11}-h_{\alpha\alpha}),
		\end{eqnarray*} which,  together  with \eqref{deri of H} and $\theta\in(0,\frac{\pi}{2}]$, implies		\begin{eqnarray*}
			\nabla_{\mu}\bar{H}&=&-(b^{11})^{2}h_{11;\mu}-\sum_{\alpha=2}^n (b^{\alpha\alpha})^2 h_{\alpha\alpha;\mu}\\
			&=&\sum_{\alpha=2}^n\left[(b^{11})^{2}-(b^{\alpha\alpha})^{2}\right]h_{\alpha\alpha;\mu}\\
			&=&\cot\theta \sum_{\alpha=2}^n h_{\alpha\alpha}(h_{11}-h_{\alpha\alpha})\left[(b^{11})^{2}-(b^{\alpha\alpha})^{2}\right] \\&\leq& 0.
		\end{eqnarray*}
	\end{proof}

	\subsection{Gradient estimates and the convergence}\label{sec3.1}

	One can show the short-time existence of flow \eqref{flow with capillary} for any capillary hypersurface by using  the method given in \cite{HP1999}. Since we consider here only for capillary hypersurfaces of star-shaped, i.e., $\<x,\nu\>>0$ on $\S_0$, we can reduce flow to a scalar flow, i.e., \eqref{scalar flow with capillary},
	and hence the short-time existence follows easily.  For the convenience of the reader and for our use later, we present the reduction argument here.

	Assume that  a capillary hypersurfaces 
	$\S\subset \bar \RR^{n+1}_+$ is  strictly star-shaped with respect to the origin. 
	One can reparametrize it  as a graph over $\overline{\mathbb{S}}^n_+$. Namely,
	there exists a positive smooth function $\rho\in C^\infty$ defined on $\overline{\SS}^n_+$ such that
	\begin{eqnarray*}
		\S=\left\{ \rho (X )X |    X\in \overline{\SS}^n_+\right\},
	\end{eqnarray*}where $X:=(X_1,\ldots, X_n)$ is a local coordinate of $\overline{\SS}^n_+$. 

We  denote  $\n^0$ as the Levi-Civita connection on $\SS^n_+$ with respect to the standard round metric $\sigma:=g_{_{\SS^n_+}}$, $\p_i:=\p_{X_i}$,  $\sigma_{ij}:=\sigma(\p_{i},\p_{j})$,  $\rho_i:=\n^0_{i} \rho$, and $\rho_{ij}:=\n^0_{i}\n^0_j \rho$. 
Moreover we denote $\p_\beta$ the unit outward normal of $\p \SS^n_+$ in $\overline \SS^n_+$. The induced metric $g$ on $\S$ is given by
\begin{eqnarray*}
	g_{ij}=\rho^2\sigma_{ij}+ \rho_i\rho_j=e^{2\varphi}\left(\sigma_{ij}+\varphi_i\varphi_j\right),
\end{eqnarray*}where  $\varphi(X):=\log \rho(X)$. Its inverse $g^{-1}$ is given by
\begin{eqnarray*}
	g^{ij}=\frac{1}{\rho^2} \left(\sigma^{ij}-\frac{\rho^i\rho^j}{\rho^2+|\n^0 \rho|^2}\right)=e^{-2\varphi}\left( \sigma^{ij}-\frac{\varphi^i\varphi^j}{v^2}\right),
\end{eqnarray*}where $\rho^i:=\sigma^{ij}\rho_j$,  $\varphi^i:=\sigma^{ij}\varphi_j$ and $v:=\sqrt{1+|\n^0 \varphi|^2}$.
The unit outward normal vector field  on $\S$ is given by
\begin{eqnarray*}
	\nu=\frac{1}{v}\left( \p_\rho-\rho^{-2}\n^0  \rho \right)=\frac{1}{v}\left( \p_\rho-\rho^{-1}\n^0 \varphi\right).
\end{eqnarray*}
Hence the capillary condition, i.e. $\Sigma$ intersects with $\p\R^{n+1}_+$ at an angle $\theta$, is
expressed by
\begin{eqnarray*}
-\cos\theta=\<\nu, \frac{1}{\rho} \p_\beta\>=- \frac 1 v  \n^0_{\p_\beta }\varphi,
\end{eqnarray*}
which is equivalent to
$$\n^0_{\p_\beta }\varphi =\cos\theta \sqrt {1+|\n^0\varphi|^2}.
$$
The second fundamental form $h$ on $\S$ is
\begin{eqnarray*}
	h_{ij}=\frac{e^\varphi}{v }\left( \sigma_{ij}+\varphi_i\varphi_j-\varphi_{ij}\right),
\end{eqnarray*}and its Weingarten matrix $(h^i_j)$ is
\begin{eqnarray*}
	h_j^{i}=g^{ik}h_{kj}=\frac{1}{e^\varphi v }\left[ \delta^i_j-(\sigma^{ik}-\frac{\varphi^i\varphi^k}{v^2})\varphi_{kj}\right].
\end{eqnarray*}
We also have for the mean curvature,
\begin{eqnarray}
	H=\frac{1}{e^\varphi v}\left[ n-\left(\s^{ij}-\frac{\varphi^i\varphi^j}{v^2} \right)\varphi_{ij}\right].
\end{eqnarray}

Now under the assumption of star-shapedness, the 
 flow \eqref{flow with capillary} 
 can be reduced to %of radial graphs over $\bar \SS^n_+$, and it is known that 
 a  scalar  evolution equation for  $\varphi(z,t)=\log \rho (X(z,t),t)$, 
 \begin{eqnarray}\label{scalar flow with capillary}
	\begin{array}{rcll}
		\p_t \varphi &=&F,\quad &\text{ in } \SS^n_+\times [0,T^*),  \\
		\n_{\p_\beta}^0 \varphi&=&\cos\theta \sqrt{1+|\n^0\varphi|^2},\quad &\text{ on }
		\p\SS^n_+\times [0,T^*),\\
		\varphi(\cdot,0)&=&\varphi_0(\cdot),\quad &\text{ on } \SS^n_+,
	\end{array}
\end{eqnarray}
where $\varphi_0 $ is the parameterization radial function of $x_0(M)$ over $\overline{\SS}^n_+$, and
\begin{eqnarray*}
	F:=  \frac{1}{e^\varphi}\div_{_{\SS^n_+}}\left( \frac{\n^0\varphi}{v}\right) +\frac{v}{e^\varphi} \left[1- \frac{\cos\theta}{v}\left( \cos\beta+\sin\beta \n^0_{\p_\beta}\varphi \right)  -\frac{n}{v^2}\right].%-\frac{e^\varphi}{v}H.
\end{eqnarray*}
The short-time existence of this scalar flow follows from the standard PDE theory. It is clear that it provides 
the short-time existence of flow \eqref{flow with capillary}. Now assume $T^*>0$ is 
 the maximal time of  existence of a solution to \eqref{flow with capillary},
	in the class of star-shaped hypersurfaces.

 The star-shapedness of $\S_0$ implies that there exist some $0<r_1<r_2<\infty$, such that 
 $$\S_0\subset \widehat{C_{\theta,r_2}}\setminus \widehat{C_{\theta,r_1}}.$$
Following the argument in \cite[Proposition 4.2]{WWX2022}, which is based on the avoidable principle, we have 
\begin{prop}\label{c0 est}
	For any $t\in [0, T^*)$, along flow \eqref{flow with capillary}, it holds
	 \begin{eqnarray}\label{c0 bound}
	 	\Sigma_{t}\subset \widehat{C_{\theta,r_2}}\setminus\ \widehat{C_{\theta,r_1}}.
	 \end{eqnarray}
	 where $C_{\theta,r}$ defined by \eqref{sph-cap} and $r_{1}, r_{2}, c_{0}$ only depend on $\Sigma_{0}$.
\end{prop} 
	
Next we show the star-shapedness is preserved along flow \eqref{flow with capillary}, which follows  the uniform gradient estimate of $\varphi$ in \eqref{scalar flow with capillary}.
\begin{prop}	\label{c1-est}
	Let $\S_0$ be a star-shaped hypersurface with capillary boundary in $\bar \RR^{n+1}_+$, and $\theta \in(0,\pi)$, then there exists $c_0>0$ depending only on $\S_0$, such that 
	\begin{eqnarray}\label{c1 est}
		\langle x,\nu\rangle(p,t) \geq c_0.
	\end{eqnarray}for all $(p,t)\in M\times [0,T^*)$.
\end{prop}

\begin{proof}
From Proposition \ref{evo of relative u} and  $|h|^2\geq \frac{H^2}{n}$, we see
	\begin{eqnarray*}
		\L \bar u&=&  
		n-\frac{2uH}{1+\cos\theta \<\nu,e\>}+\frac{u^2 |h|^2}{(1+\cos\theta\<\nu,e\>)^2}	 \quad \text{ mod} \quad \n \bar u	\\&\geq & \left( \frac{u|H|}{\sqrt{n}(1+\cos\theta \<\nu,e\>)} -\sqrt{n}\right)^2\geq 0,
	\end{eqnarray*}	 together with  with   $\n_\mu \bar u=0$ on $\p M$, it implies	\begin{eqnarray*}
		\bar u\geq \min_{M} \bar u(\cdot,0).
	\end{eqnarray*}
	Since  $1+\cos\theta \<\nu,e\>\geq 1-\cos\theta >0$, we conclude 
	\begin{eqnarray*}
		\<x,\nu\>=\bar u\cdot (1+\cos\theta \<\nu,e\>)\geq c_0,
	\end{eqnarray*} for some positive constant $c_0>0$ independent of $t$.
\end{proof}

\begin{prop}\label{mcf convegence}
	If the initial hypersurface $\S_0\subset  \bar \RR^{n+1}_+$ is star-shaped, then flow \eqref{flow with capillary} exists for all time. Moreover, it smoothly converges to a uniquely determined spherical cap $C_{\theta, r}(e)$ given by \eqref{sph-cap} with capillary boundary, as $t\rightarrow +\infty$.
\end{prop}
\begin{proof}
From \eqref{c0 bound} we have a uniform bound for $\varphi$, and  from  \eqref{c1 est}, we have a uniform bound for $v$, and hence a  bound for  $\nabla ^0\varphi$. Therefore
$\varphi$ is uniformly bounded in $C^1(\bar\SS^n_+\times[0,T^*))$ and the scalar equation in \eqref{scalar flow with capillary} is uniformly parabolic. Since $|\cos\theta|<1$, from the standard quasi-linear parabolic theory with a strictly oblique boundary condition theory (cf. \cite{LSU,Gary1996}), we conclude the uniform $C^\infty$ estimates and the long-time existence. And the	  convergence can be shown similarly by using the argument as in \cite{SWX,WeX21,WWX2022}, we omit it here.
\end{proof}	

\subsection{Preserving Convexity}\label{sec3.2}
In this subsection, we show flow \eqref{flow with capillary} preserves the convexity and finish the proof of Theorem \ref{thm1.4}. 	First, we show the uniform upper bound of $H$ along flow \eqref{flow with capillary}.
	\begin{prop}\label{upper bound of H}
	If $\Sigma_{t}$ solves flow \eqref{flow with capillary}, then the mean curvature is uniformly bounded from above, that is,		\begin{eqnarray*}
			H(p, t)\leq \max_{M}H(\cdot, 0),\quad\quad \forall (p, t)\in M\times [0,T^*).
		\end{eqnarray*}
	\end{prop}
	\begin{proof}
	From equation \eqref{evo of H} and $n|h|^2\geq H^2$, we see 
		\begin{eqnarray}
			\L H\leq 0,\quad \text{ mod } \n H,
		\end{eqnarray}and $\n_{\mu }H =0$ on $\p M$, hence the conclusion follows directly from  the maximum principle.%, we get the upper bounds for $H$.
	\end{proof}
Next we show the uniform lower bound of mean curvature, that is, the mean convexity is preserved along flow \eqref{flow with capillary}.
	\begin{prop}\label{lower bound of H}
	If $\Sigma_{t}$ solves flow \eqref{flow with capillary} with an initial hypersurface $\Sigma_{0}$ being a strictly mean convex capillary hypersurface   in $\bar \RR^{n+1}_+$, then
		\begin{eqnarray*}
			H(p, t)\geq C,\quad\quad \forall (p, t)\in M\times [0,T^*),
		\end{eqnarray*}
		where the positive constant  $C$ depends on the initial datum.% and $\cos\theta$.
	\end{prop}
	\begin{proof}
Define the function $$P:=H \bar u.$$ Using \eqref{neumann of bar u} and \eqref{deri of H},   it is easy to see that \begin{eqnarray}\label{neuman of P}
	 \nabla_{\mu}P= 0,\quad \text{ on } \p M. 
\end{eqnarray}
Using  \eqref{evo of bar u} and \eqref{evo of H}, we obtain
		\begin{eqnarray} \label{evolution P} 
		\L P&=& \bar u\mathcal{L}H+H\mathcal{L} \bar u-2u\<\n \bar u,\n H\> \\ \notag
			&=& \bar u^{2}H|h|^{2}-\bar uH^{2}-n\bar u |h|^{2}+nH+2\bar u\< \n(1+\cos\theta \<\nu,e\>),\n P\>.\notag
		\end{eqnarray} 
	From \eqref{neuman of P} and  the Hopf Lemma,	if $P$ attains the minimum value at $t=0$, then the conclusion follows directly by combining with the uniform bound of $\bar u$. Therefore, we assume that
 $P$ attains the minimum value at some interior point, say $p_0\in \text{int}(M)$. At $p_0$, we have
	$$\n P=0, \quad \L P\leq 0.$$ Substituting  it into \eqref{evolution P} yields
	\begin{eqnarray}\label{2nd polynoimial}
	\left( \frac n H-\bar u\right) \bar u|h|^2+\bar uH\geq n.
\end{eqnarray}	 
If $\bar uH\geq n$ at $p_0$, then  we are done. Assume now  that $\bar uH<n$ at $p_0$, then
using $|h|^2 \leq n H^2$ in \eqref{2nd polynoimial}, we obtain 
		\begin{eqnarray}
		\bar	uH(p_0, t)\geq c,
		\end{eqnarray}for some positive constant $c$, which only depends  on $n$. This also yields the desired estimate.
	
From above discussion, we complete the proof of Proposition \ref{lower bound of H}.
	\end{proof}

We now show that $\V_{k,\theta}(\widehat{\S_t})$ is 
non-increasing under  flow \eqref{flow with capillary}.
\begin{prop}\label{mono along mcf}
	As long as  flow  \eqref{flow with capillary} exists and $ \S_t$ is strictly convex, the enclosed volume	$ \V_{0,\theta}(\widehat{\S_t})$ is preserved and $ \V_{k,\theta}(\widehat{\S_t})$ is non-increasing  for $1\leq k  \leq n$.
\end{prop}

\begin{proof}
	Using  \cite[Theorem 2.6]{WWX2022} and Minkowski formula \eqref{minkowski formula-0}, we see
	\begin{eqnarray*}
		\p_t \V_{0 ,\theta}(\widehat{\S_t})=\int_{\S_t} f dA_t=0,
	\end{eqnarray*}
	and
	\begin{eqnarray*}
		\p_t \V_{k ,\theta}(\widehat{\S_t})&=&\frac{n+1-k}{n+1}\int_{\S_t} f H_{k } dA_t
		\\&=& \frac{n+1-k}{n+1}\int_{\S_t}\left[ nH_k\left(1+\cos\theta\<\nu, e\>\right)-HH_k\<x, \nu\>\right]dA_t
		\\&\leq& \frac{n(n+1-k)}{n+1}\int_{\S_t} \left[H_{k}\big(1+\cos\theta \langle \nu,e\rangle\big)-H_{k+1}\langle x,\nu\rangle\right]dA_t
		\\&=&0,
	\end{eqnarray*} where we have used the Newton-MacLaurin inequality $H_1H_k\geq H_{k+1}$, Proposition \ref{c1-est} and the  Minkowski formula \eqref{minkowski formula} in the last two steps.
	
\end{proof}

The proof implies the following nice property, i.e. a characterization result on the spherical cap.

\begin{prop}\label{prop_0}
If a capillary hypersurface $\Sigma$ satisfies \eqref{eq_a1}, i.e.,
\begin{eqnarray}\label{eq_a2}
	1+\cos\theta\<\nu, e\>-\frac{H}{n}\<x, \nu\>=0,
\end{eqnarray} 
then it is $C_{r,\theta}$ for some $r>0$.

\end{prop}

	\begin{proof}
		Since $|\cos \theta| <1$, \eqref{eq_a2} implies that 
		$\<x,\nu\>H\ge c_0>0$ on $\Sigma$, which means that $H$ can never change sign.
		 It is clear  that $\S$ is contained in a $C_{r, \theta}$ with $r$ large.
		Then we decrease $r$, there exists a  first largest radius $r_0$ and a first touch  point $x_0\in \Sigma$.
		If $x_0$ lies in the interior of $\S$, then  $C_{r_0,\theta}$ touches $\Sigma$ at $x_0$ tangentially. If $x_0 \in \p \S$, then due to the assumption of the contact angle
		of $\Sigma$,   $C_{r_0,\theta}$ touches also $\Sigma$ at $x_0$ tangentially. It follows that $H(x_0)>0$ , and hence $H>0$ and $\<x,\nu\> >0$ on $\Sigma$. 
		Now from \eqref{eq_a2}  and $H^2 =n^2 H_1^2 \ge n^2 H_2$, we have
			\begin{eqnarray*}
		0&=& \int_{\S}\left[ nH\left(1+\cos\theta\<\nu, e\>\right)-H^2\<x, \nu\>\right]dA
			\\&\leq& n^2\int_{\S} \left[H_1\big(1+\cos\theta \langle \nu,e\rangle\big)-H_{2}\langle x,\nu\rangle\right]dA
			\\&=&0,
		\end{eqnarray*} 
 the equality  implies that $\Sigma$ is umbilical, which follows also the conclusion.
		\end{proof}
Now  we prove that the mean curvature type flow \eqref{flow with capillary} preserves the convexity, if $\theta\in (0,\frac{\pi}{2}]$. 
	\begin{prop}\label{H case preserve convexity}
		Let $\Sigma_{0}$ be a strictly convex hypersurface, if $\Sigma_{t}$ solves flow \eqref{flow with capillary} with the initial value $\Sigma_{0}$ and  $\theta\in (0,\frac{\pi}{2}]$, then
		\begin{eqnarray*}
			\min_{1\leq i\leq n}\kappa_{i}(p, t)\geq c,\quad \quad \forall (p, t)\in M\times[0, T^*),
		\end{eqnarray*}
		where the positive constant $c$ depends only on $\Sigma_{0}$.  
	\end{prop}
	\begin{proof}
It is equivalent to  show the uniform upper bound for $\bar{H}$. 
If the maximum value of $\bar H$ is reached at $t=0$, then we are done. Otherwise, from $\n_{\mu}\bar H\leq 0$ on $\p M$ in Proposition \ref{evo of bar H} and the Hopf boundary Lemma, we have that $\bar{H}$   attains its maximum value at some interior point, say $p_{0}\in {\rm int}(\Sigma_t)$. At $p_0$, we choose an orthonormal frame $\{e_{i}\}_{i=1}^{n}$ such that $(h_{ij})$ is diagonal, which follows that $(b^{ij})$ is also  diagonal. Noticing that for each fixed $k$, it holds
		\begin{eqnarray*}
\sum_{i,j}	  b^{jj} h_{ij;k}^2	&=&\sum_{j}\frac{1}{H}\left(1+\sum_{l\neq j} \frac{h_{ll}}{h_{jj}}\right)\left(h_{jj;k}^2+\sum_{i\neq j} h_{ij;k}^2\right)
\\&\geq &  \frac{1}{H} \left( \sum_jh_{jj;k}^2 + \sum_{i\neq j} \frac{h_{ii}}{h_{jj}} h_{jj;k}^2\right)
\\&= &\frac{1}{H} \left[ \sum_jh_{jj;k}^2 + \sum_{i>j} \left( \frac{h_{ii}}{h_{jj}} h_{jj;k}^2+\frac{h_{jj}}{h_{ii}} h_{ii;k}^2\right)\right]
\\&\geq & \frac{H_{;k}^2}{H},
	\end{eqnarray*}where the last inequality follows from the Cauchy-Schwarz inequality. Substituting this into \eqref{evo of widetilde H}, at $p_0$,  we have 
		\begin{eqnarray*}
			0&\leq &\mathcal{L}\bar{H} %\\			&=&-2u  (b^{ii})^{2}\cdot b^{jj}h_{jk;i}^{2}-2b^{ii}H_{;i}\<x, e_{i}\>-u|h|^{2}\bar{H}-H\bar{H}+n^{2}+nuH
			\\
          	&\leq &-2H^{-1}u(b^{ii})^{2}H_{;i}^{2}-2b^{ii}H_{;i}\<x, e_{i}\>-u|h|^{2}\bar{H}-H\bar{H}+n^{2}+nuH\\
			&\leq &\frac{Hu^{-1}|x^T|^2}{2}-u|h|^{2}\bar{H}-H\bar{H}+n^{2}+nuH,
		\end{eqnarray*} which
together with Proposition \ref{lower bound of H}, implies  $\bar{H}\leq C$,  the desired estimate. Hence we complete the proof.
	\end{proof}

\begin{proof}[\textbf{Proof of Theorem \ref{thm1.4}}]\
	
 	Assume that $\S$ is strictly convex, the proof of Theorem \ref{thm1.4} follows from our main Theorem \ref{thm 1.1}, the monotonicity of
the relative quermassintegrals, Proposition \ref{mono along mcf}  and Proposition \ref{H case preserve convexity}.  

	When $\S$ is convex but not strictly convex, the inequality \eqref{af ineq V_k V_0} follows by approximation. The equality characterization can be proved  similar to \cite[Section 4]{SWX}, by using an argument of \cite{GL09}. We omit the details here.
\end{proof}

\
\

\noindent\textbf{Acknowledgment:} XM is partially  supported by CSC (No. 202106340053) and the doctoral dissertation creation project of USTC. LW is partially supported by NSFC (Grant No. 12201003, 12141105, 12171260). We would like to thank the referee for  the  careful  reading and valuable suggestions to improve the context of the paper.

\end{document}